\documentclass[12pt,a4paper]{amsart}
\usepackage{amsmath,amsthm,amsfonts}
\usepackage{hyperref}
\usepackage{graphicx}
\usepackage{tikz}
\usepackage{caption}
\usepackage{subcaption}

\newtheorem{theorem}{Theorem}
\newtheorem{lemma}{Lemma}
\newtheorem{corollary}{Corollary}
\begin{document}
\title{Erratum for Ricci-flat graphs with girth at least five}
\date{\today}
\author[D. Cushing]{David Cushing}
\address{D. Cushing,  Department of Mathematical Sciences, Durham University, Durham DH1 3LE, UK}
\email{david.cushing@durham.ac.uk}

\author[R. Kangaslampi]{Riikka Kangaslampi}
\address{R. Kangaslampi,  Computing sciences, Faculty of information technology and communication sciences,
Tampere University,  FI-33014 Tampere, Finland
}
\email{riikka.kangaslampi@tuni.fi}

\author[Y. Lin]{Yong Lin}
\address{Y. Lin, Yau Mathematical Sciences Center, Tsinghua University, Beijing 100084, China}
\email{yonglin@mail.tsinghua.edu.cn}

\author[S. Liu]{Shiping Liu}
\address{S. Liu, School of Mathematical Sciences, University of Science and Technology of China, Hefei 230026, China}
\email{spliu@ustc.edu.cn}

\author[L. Lu]{Linyuan Lu}
\address{L. Lu, Department of Mathematics, University of South Carolina, Columbia, SC 29208, USA}
\email{lu@math.sc.edu}

\author[S.-T. Yau]{Shing-Tung Yau}
\address{S.-T. Yau, Department of Mathematics, Harvard University, Cambridge, MA 02138, USA}
\email{yau@math.harvard.edu}
\begin{abstract}
   This erratum will correct the classification of Theorem 1 in \cite{LLY14} that misses the Triplex graph.
\end{abstract}
\maketitle
In Theorem 1 of \cite{LLY14}, the classification of Ricci-flat graph with girth $g(G)\geq 5$ missed one graph -- the Triplex graph, as discovered by three authors: Cushing, Kangaslampi, and Liu.
Here is the correct theorem.
\begin{theorem}\label{t1}
Suppose that $G$ is a Ricci-flat graph with girth $g(G)\geq 5$. Then
$G$ is one of the following graphs,
\begin{enumerate}
\item the infinite path,
\item  cycle $C_n$ with $n\geq 6$,
\item the  dodecahedral graph,
\item the Petersen graph,
\item the half-dodecahedral graph.
\item the Triplex graph.
\end{enumerate}
\end{theorem}

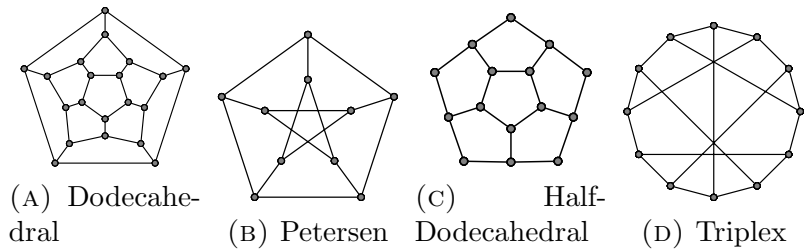
\begin{figure}[h!]
\begin{center}
\tikzstyle{every node}=[circle, draw, fill=black!50,
inner sep=0pt, minimum width=4pt]
    \begin{subfigure}[b]{0.2\textwidth}
        \centering
        \resizebox{\linewidth}{!}{

 \begin{tikzpicture}[thick,scale=0.6]%
       \draw \foreach \x in {18,90,...,306} {
        (\x+36:1)  -- (\x-36:1)
        (\x:3.5)  -- (\x-72:3.5)
        (270:1) node[label=right:$$](v){}
		(342:1) node[label=right:$$](u){}
		(198:1) node[label=left:$$](w){}
		(54:1) node[label=right:$$](y){}
		(126:1) node[label=left:$$](x){}
	        (270:1.7) node[label=right:$$](v1){}
		(342:1.7) node[label=right:$$](u1){}
		(198:1.7) node[label=left:$$](w1){}
		(54:1.7) node[label=right:$$](y1){}
		(126:1.7) node[label=left:$$](x1){}
	        (306:2.5) node[label=right:$$](v2){}
		(378:2.5) node[label=right:$$](u2){}
		(234:2.5) node[label=left:$$](w2){}
		(90:2.5) node[label=right:$$](y2){}
		(162:2.5) node[label=left:$$](x2){}	
		(306:3.5) node[label=right:$$](v3){}
		(378:3.5) node[label=right:$$](u3){}
		(234:3.5) node[label=left:$$](w3){}
		(90:3.5) node[label=right:$$](y3){}
		(162:3.5) node[label=left:$$](x3){}			
        };
        \draw (v)--(v1) (u)--(u1) (w)--(w1) (x)--(x1) (y)--(y1);
        \draw (y1)--(u2)--(u1)--(v2)--(v1)--(w2)--(w1)--(x2)--(x1)--(y2)--(y1);
        \draw (v2)--(v3) (u2)--(u3) (w2)--(w3) (x2)--(x3) (y2)--(y3);

\end{tikzpicture}
}
\caption{\small Dodecahedral}
\label{fig:dodecahedron}
\end{subfigure}
\begin{subfigure}[b]{0.2\textwidth}
    \centering
        \resizebox{\linewidth}{!}{
 \begin{tikzpicture}[thick,scale=1]%
        \draw \foreach \x in {18,90,...,306} {
        (\x:2)  -- (\x-72:2)

         (90:2) node[label=right:$$](v){}
	 (162:2) node[label=right:$$](u){}
	 (18:2) node[label=left:$$](w){}
	 (234:2) node[label=right:$$](y){}
	 (306:2) node[label=left:$$](x){}
	 (90:1) node[label=right:$$](v1){}
	 (162:1) node[label=right:$$](u1){}
	 (18:1) node[label=left:$$](w1){}
	 (234:1) node[label=right:$$](y1){}
	 (306:1) node[label=left:$$](x1){}
	};	
\draw (x)--(x1) (y)--(y1) (u)--(u1) (v)--(v1) (w) --(w1);
\draw (x1)--(u1) (y1)--(v1) (u1) --(w1) (v1)--(x1) (w1)--(y1);
\end{tikzpicture}
}
\caption{\small Petersen}
\label{fig:petersen}
    \end{subfigure}
    \begin{subfigure}[b]{0.2\textwidth}
        \centering
        \resizebox{\linewidth}{!}{
 \begin{tikzpicture}[thick,scale=0.6]%
       \draw \foreach \x in {18,90,...,306} {
        (\x+36:1)  -- (\x-36:1)
        (270:1) node[label=right:$$](v){}
		(342:1) node[label=right:$$](u){}
		(198:1) node[label=left:$$](w){}
		(54:1) node[label=right:$$](y){}
		(126:1) node[label=left:$$](x){}
	        (270:2.05) node[label=right:$$](v1){}
		(342:2.05) node[label=right:$$](u1){}
		(198:2.05) node[label=left:$$](w1){}
		(54:2.05) node[label=right:$$](y1){}
		(126:2.05) node[label=left:$$](x1){}
	        (306:2.5) node[label=right:$$](v2){}
		(378:2.5) node[label=right:$$](u2){}
		(234:2.5) node[label=left:$$](w2){}
		(90:2.5) node[label=right:$$](y2){}
		(162:2.5) node[label=left:$$](x2){}	
        };
        \draw (v)--(v1) (u)--(u1) (w)--(w1) (x)--(x1) (y)--(y1);
        \draw (y1)--(u2)--(u1)--(v2)--(v1)--(w2)--(w1)--(x2)--(x1)--(y2)--(y1);
\end{tikzpicture}
}
\caption{\small Half-Dodecahedral}
\label{fig:halfdodecahedron}
\end{subfigure}
   \begin{subfigure}[b]{0.2\textwidth}
        \centering
        \resizebox{\linewidth}{!}{

 \begin{tikzpicture}[thick,scale=1]%
 \draw \foreach \x in {0,30,...,330} {
        (\x:2)  -- (\x+30:2)
		(90:2) node(x1){}
		(60:2) node(x2){}
		(30:2) node(x3){}
		(0:2) node(x4){}
		(-30:2) node(x5){}
		(-60:2) node(x6){}
		(120:2) node(x12){}
		(150:2) node(x11){}
		(180:2) node(x10){}
		(210:2) node(x9){}
		(240:2) node(x8){}
		(270:2) node(x7){}		
		};
 \draw (x1)--(x7) (x2)-- (x10) (x3)--(x8) (x4)--(x12) (x5)-- (x9) (x6)-- (x11);		
\end{tikzpicture}
}
\caption{\small Triplex}
\label{fig:triplex}
    \end{subfigure}
   \end{center}
   \caption{The four Ricci-flat graphs with girth 5}
   \label{fig:result}
\end{figure}

This error was caused by an incorrect implicit statement (in \cite{LLY14}) that any $3$-regular Ricci-flat graph
$G$ has a surface embedding whose faces are all pentagons. In this erratum,
we analyze the case that $G$ does not have a surface embedding whose faces are all pentagons.
We will show that this case leads a unique missing graph --- the Triplex graph. An alternative method to correct Theorem 1 in \cite{LLY14} is given in \cite{CKLLLY18}.

Recall that Lemma 3 item 2 in \cite{LLY14} states:
\begin{lemma}
  \label{local_structure} For any edge $xy$ of a graph of girth at least $5$,
  if $d_x=d_y=3$ and  $\kappa(x,y)=0$, then $xy$ belongs to two $5$-cycles $P_1$ and $P_2$ such that $P_1\cap P_2=xy$.
\end{lemma}

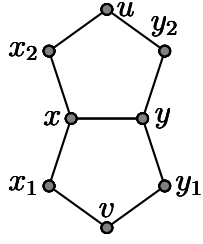
\begin{figure}[!ht]
\begin{center}
\tikzstyle{every node}=[circle, draw, fill=black!50,
                        inner sep=0pt, minimum width=4pt]
 \begin{tikzpicture}[thick,scale=0.8]%
 \draw \foreach \x in {18,90,...,306} {
        (\x:1)  -- (\x+72:1)
		(18:1) node[label=above:$y_2$]{}
		(162:1) node[label=left:$x_2$]{}
		(90:1) node[label=right:$u$]{}
		};
   \begin{scope}[shift={(0,-1.618)}]
       \draw \foreach \x in {18,90,...,306} {
        (\x+36:1)  -- (\x-36:1)
        (270:1) node[label=above:$v$]{}
		(342:1) node[label=right:$y_1$]{}
		(198:1) node[label=left:$x_1$]{}
		(54:1) node[label=right:$y$]{}
		(126:1) node[label=left:$x$]{}
        };
  \end{scope}
\end{tikzpicture}
\end{center}
\end{figure}

Since $G$ contains no cycle of length $3$ or $4$, any $C_5$ containing the edge $xy$
is uniquely determined by a $3$-path passing through $xy$. Since $d_x=d_y=3$,
there are four $3$-paths of form $x_ixyy_j$ for $i,j=1,2$. Here $x_1,x_2$
are two neighbors of $x$ other than $y$ and $y_1, y_2$ are two neighbors of $y$
other than $x$.
We say two $C_5$'s
are {\em opposite} to each other at $xy$ if one $C_5$ passes through $x_ixyy_j$
and the other one passes through $x_{3-i}xyy_{3-j}$.
The above lemma says that there is a pair of opposite $C_5$'s sharing the edge $xy$.
We say an edge $xy$
is {\em irregular} if there are three or four $C_5$ passing through it.

From this lemma, we have the following corollary.
\begin{corollary}
  If $G$ is a $3$-regular Ricci-flat graph and contains no irregular edge,
then $G$ can be embedded into a surface so that all faces are pentagons.
\end{corollary}
\begin{proof}
  View $G$ as $1$-dimension skeleton and glue pentagons to $G$ recursively.
  Starting with any $C_5$ and glue a pentagon to it as a face, call the two-dimensional region $M$.

Let $xy$ be a boundary edge of $M$, that is, an edge belonging to only one pentagon $f$ in $M$. This pentagon $f$ determines an opposite $C_5$ of $G$ with respect to the edge $xy$. We glue a pentagon face to the opposite
$C_5$ at $xy$ to enlarge $M$. Since $G$ contains no irregular edge, every edge must be in exactly two pentagons. Therefore, the process
will continue until $M$ has no boundary edge. When this process ends,
we get an embedding of $G$ into some surface so that every face is a $C_5$.
\end{proof}

We are ready to fix the proof of Theorem 1 in \cite{LLY14}.

\begin{proof}[Proof of Theorem \ref{t1}:]
  In the original proof of Theorem 1 in \cite{LLY14}, we have taken
  care of all the cases except that $G$ is $3$-regular and contains
  an irregular edge $xy$. The edge $xy$ is either in three $C_5$'s or four $C_5$'s. We will show that the first case leads to the Triplex graph while the second case leads to the Petersen graph.

  \begin{figure}[!ht]
\begin{center}
\tikzstyle{every node}=[circle, draw, fill=black!50,
                        inner sep=0pt, minimum width=4pt]
 \begin{tikzpicture}[thick,scale=0.8]%
 \draw \foreach \x in {18,90,...,306} {
        (\x:1)  -- (\x+72:1)
		(18:1) node[label=above:$y_2$](y2){}
		(162:1) node[label=left:$x_2$](x2){}
		(90:1) node[label=right:$u$](u){}
		};
   \begin{scope}[shift={(0,-1.618)}]
       \draw \foreach \x in {18,90,...,306} {
        (\x+36:1)  -- (\x-36:1)
        (270:1) node[label=above:$v$](v){}
		(342:1) node[label=right:$y_1$](y1){}
		(198:1) node[label=below:$x_1$](x1){}
		(54:1) node[label=right:$y$](y){}
		(126:1) node[label=left:$x$](x){}
        };
      \end{scope}
      \draw (-3.2,-1.8) node[label=left:$w$] (w){};
      \draw (-2,-1.8) node[label=below:$w_1$] (w1){};
      \draw (2,0) node[label=right:$w_2$] (w2){};
      \draw (w)-- (x2);
      \draw (w1)-- (x1);
      \draw (w2)-- (y2);
      \draw (w) to [out=-60,in=-120, distance=2.5cm] node[draw=none,fill=none,right] {} (y1) ;
      \draw[dashed] (w) to [out=0,in=180 ] node[draw=none,fill=none,below] {i)} (w1);
      \draw[dashed] (w1) to [out=90,in=180, distance=2.5cm] node[draw=none,fill=none,left] {ii)} (u) ;
      \draw[dashed] (w) to [out=-90,in=-70, distance=4.5cm] node[draw=none,fill=none,right] {iii)} (w2) ;
      \draw[dashed] (v) to [out=0,in=-90, distance=2cm] node[draw=none,fill=none,below] {iv)} (w2);
    \end{tikzpicture}
    \caption{Starting configuration and possible extensions}
    \label{start}
  \end{center}
\end{figure}
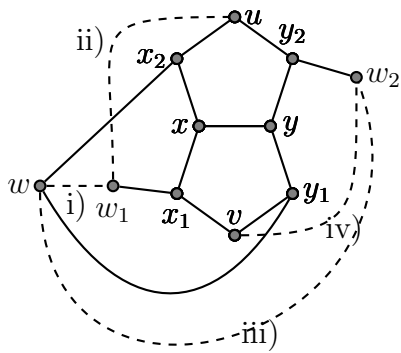
First assume the edge $xy$ is contained in three $C_5$'s:  $ux_2xyy_2u$,  $vx_1xyy_1v$, and $wx_2xyy_1w$.
The path $x_1xyy_2$ is not in any $C_5$. Let $w_1$ be the third neighbor of $x_1$, and $w_2$ be the third neighbor
of $y_2$. Then $w_1$, $w_2$ are two distinct vertices, and they cannot be coincident with any vertex on the three
$C_5$'s. This is our starting configuration (See Figure \ref{start} with solid lines).

Now consider the edge $xx_1$. Observe that the path $w_1x_1xy$ is not on any $C_5$. Thus, the path $w_1x_1xx_2$ must
be extended to a $C_5$. Either $w_1u$ is an edge or $w_1w$ is an edge. Similarly, by considering the edge $yy_2$,
either $w_2w$ or $w_2v$ is an edge. These four possible edges are shown as dashed lines i), ii), iii), and iv) in Figure \ref{start}.
There are four combinations: i)+iii), i)+iv), ii)+iii), ii)+iv).
The combination i)+iii) is impossible since $d_w=3$. The two cases i)+iv) and ii)+iii) are symmetric.
Essentially we have two cases to consider:
\begin{description}
\item[Case] i)+iv): Now consider the edge $w_1x_1$. By Lemma 1, there are a pair of opposite $C_5$ sharing
  the edge $w_1x_1$. Such a pair of opposite pentagons can be obtained only by adding a new vertex $w_3$ as the third neighbor of $w_1$ and connecting $w_3$ to $w_2$, since connecting $w_1$ to $u$ would cause a $C_4$. Now $x_1w_1w_3w_2vx_1$ and $x_1w_1wx_2xx_1$ are the two opposite pentagons at $x_1w_1$.  But, in order to have two opposite pentagons also at the new edge $w_1w_3$ we must have $w_3u$ as an edge, which then creates a $C_4$: $w_3uy_2w_2w_3$. Contradiction!
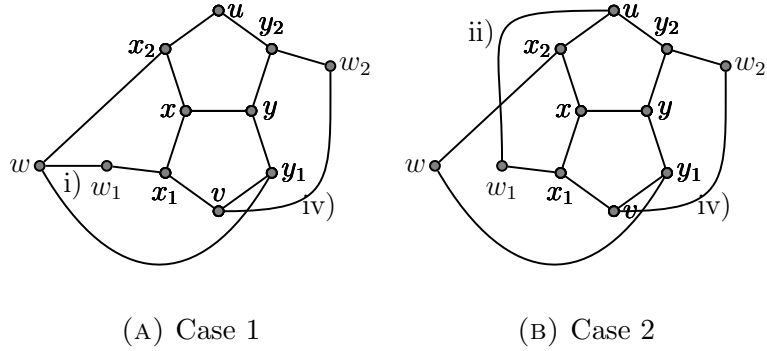
\begin{figure}[!ht]
  \begin{center}
    \tikzstyle{every node}=[circle, draw, fill=black!50,inner sep=0pt, minimum width=4pt]
    \begin{subfigure}[b]{0.4\textwidth}
      \centering
      \resizebox{\linewidth}{!}{
        \begin{tikzpicture}[thick,scale=0.8]%
          \draw \foreach \x in {18,90,...,306} {
            (\x:1)  -- (\x+72:1)
            (18:1) node[label=above:$y_2$](y2){}
            (162:1) node[label=left:$x_2$](x2){}
            (90:1) node[label=right:$u$](u){}
          };
          \begin{scope}[shift={(0,-1.618)}]
            \draw \foreach \x in {18,90,...,306} {
              (\x+36:1)  -- (\x-36:1)
              (270:1) node[label=above:$v$](v){}
              (342:1) node[label=right:$y_1$](y1){}
              (198:1) node[label=below:$x_1$](x1){}
              (54:1) node[label=right:$y$](y){}
              (126:1) node[label=left:$x$](x){}
            };
          \end{scope}
          \draw (-3.2,-1.8) node[label=left:$w$] (w){};
          \draw (-2,-1.8) node[label=below:$w_1$] (w1){};
          \draw (2,0) node[label=right:$w_2$] (w2){};
          \draw (w)-- (x2);
          \draw (w1)-- (x1);
          \draw (w2)-- (y2);
          \draw (w) to [out=-60,in=-120, distance=2.5cm] node[draw=none,fill=none,right] {} (y1) ;
          \draw (w) to [out=0,in=180 ] node[draw=none,fill=none,below] {i)} (w1);
          \draw (v) to [out=0,in=-90, distance=2cm] node[draw=none,fill=none,below] {iv)} (w2);
        \end{tikzpicture}
      }
      \caption{Case 1}
      \label{fig:case1}
    \end{subfigure}
    \begin{subfigure}[b]{0.4\textwidth}
      \centering
      \resizebox{\linewidth}{!}{
        \begin{tikzpicture}[thick,scale=0.8]%
          \draw \foreach \x in {18,90,...,306} {
            (\x:1)  -- (\x+72:1)
            (18:1) node[label=above:$y_2$](y2){}
            (162:1) node[label=left:$x_2$](x2){}
            (90:1) node[label=right:$u$](u){}
          };
          \begin{scope}[shift={(0,-1.618)}]
            \draw \foreach \x in {18,90,...,306} {
              (\x+36:1)  -- (\x-36:1)
              (270:1) node[label=right:$v$](v){}
              (342:1) node[label=right:$y_1$](y1){}
              (198:1) node[label=below:$x_1$](x1){}
              (54:1) node[label=right:$y$](y){}
              (126:1) node[label=left:$x$](x){}
            };
          \end{scope}
          \draw (-3.2,-1.8) node[label=left:$w$] (w){};
          \draw (-2,-1.8) node[label=below:$w_1$] (w1){};
          \draw (2,0) node[label=right:$w_2$] (w2){};
          \draw (w)-- (x2);
          \draw (w1)-- (x1);
          \draw (w2)-- (y2);
          \draw (w) to [out=-60,in=-120, distance=2.5cm] node[draw=none,fill=none,right] {} (y1) ;
          \draw (w1) to [out=90,in=180, distance=2.5cm] node[draw=none,fill=none,left] {ii)} (u) ;
          \draw (v) to [out=0,in=-90, distance=2cm] node[draw=none,fill=none,below] {iv)} (w2);
        \end{tikzpicture}
      }
      \caption{Case 2}
      \label{fig:case2}
    \end{subfigure}
    \caption{Two non-isomorphic ways to continue}
    \label{next}
  \end{center}
\end{figure}
\item[Case] ii)+iv). Let $w_3$ be the third neighbor of $w$. ($w_3$ is distinct from $w_1$ and $w_2$
  since the girth of $G$ is at least $5$.) Applying Lemma 1 on the edge $wx_2$, we must have a pair of opposite
  $C_5$'s passing through $wx_2$. This will force $w_3w_1$ to be an edge. Similarly, by considering $wy_1$,
  we conclude that $w_3w_2$ must be an edge. This completes a $3$-regular graph. It is easy to check
  this is the Triplex graph.

  \begin{figure}[!ht]
  \begin{center}
    \tikzstyle{every node}=[circle, draw, fill=black!50,inner sep=0pt, minimum width=4pt]
        \begin{tikzpicture}[thick,scale=0.8]%
          \draw \foreach \x in {18,90,...,306} {
            (\x:1)  -- (\x+72:1)
            (18:1) node[label=above:$y_2$](y2){}
            (162:1) node[label=left:$x_2$](x2){}
            (90:1) node[label=right:$u$](u){}
          };
          \begin{scope}[shift={(0,-1.618)}]
            \draw \foreach \x in {18,90,...,306} {
              (\x+36:1)  -- (\x-36:1)
              (270:1) node[label=above:$v$](v){}
              (342:1) node[label=right:$y_1$](y1){}
              (198:1) node[label=below:$x_1$](x1){}
              (54:1) node[label=right:$y$](y){}
              (126:1) node[label=left:$x$](x){}
            };
          \end{scope}
          \draw (-3.2,-1.8) node[label=left:$w$] (w){};
          \draw (-2,-1.8) node[label=below:$w_1$] (w1){};
          \draw (2,0) node[label=right:$w_2$] (w2){};
           \draw (-1,-3) node[label=below:$w_3$] (w3){};
          \draw (w)-- (x2);
          \draw (w1)-- (x1);
          \draw (w2)-- (y2);
          \draw (w) to [out=-60,in=-120, distance=2.5cm] node[draw=none,fill=none,right] {} (y1) ;
          \draw (w1) to [out=90,in=180, distance=2.5cm] node[draw=none,fill=none,left] {} (u) ;
          \draw (v) to [out=0,in=-90, distance=2cm] node[draw=none,fill=none,below] {} (w2);
          \draw (w) -- (w3);
          \draw (w1) -- (w3);
          \draw (w3) to [out=0,in=-70, distance=3cm] node[draw=none,fill=none,below] {} (w2);
        \end{tikzpicture}
        \caption{Unique way to complete into the Triplex graph.}
    \label{final}
  \end{center}
\end{figure}
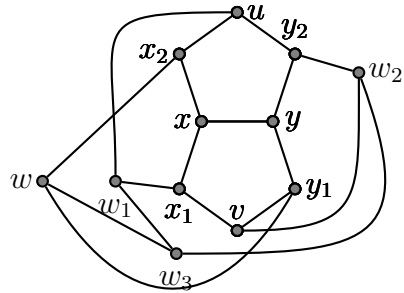
\end{description}

Now we assume $xy$ is in four $C_5$'s. For $i=1,2$ and $j=1,2$, write $k=2(i-1)+j$ and let $u_k$ be the vertex in the $C_5$ extending the path $x_ixyy_j$. Observe that connecting any pair  $u_1u_2$, $u_2u_4$, $u_4u_3$, or $u_1u_3$ results a triangle. So only $u_2u_3$ and $u_1u_4$ can be connected (See Figure \ref{start2}).

Note that $yy_1$ are in two non-opposite $C_5$'s: $xyy_1u_1x_1x$ and $xyy_1u_3x_2x$.
So either $u_2u_3$ or $u_1u_4$ must be an edge.

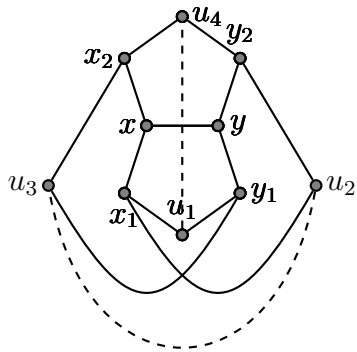
\begin{figure}[!ht]
\begin{center}
\tikzstyle{every node}=[circle, draw, fill=black!50,
                        inner sep=0pt, minimum width=4pt]
 \begin{tikzpicture}[thick,scale=0.8]%
 \draw \foreach \x in {18,90,...,306} {
        (\x:1)  -- (\x+72:1)
		(18:1) node[label=above:$y_2$](y2){}
		(162:1) node[label=left:$x_2$](x2){}
		(90:1) node[label=right:$u_4$](u4){}
		};
   \begin{scope}[shift={(0,-1.618)}]
       \draw \foreach \x in {18,90,...,306} {
        (\x+36:1)  -- (\x-36:1)
        (270:1) node[label=above:$u_1$](u1){}
		(342:1) node[label=right:$y_1$](y1){}
		(198:1) node[label=below:$x_1$](x1){}
		(54:1) node[label=right:$y$](y){}
		(126:1) node[label=left:$x$](x){}
        };
      \end{scope}
      \draw (-2.2,-1.8) node[label=left:$u_3$] (u3){};
      \draw (2.2,-1.8) node[label=right:$u_2$] (u2){};
      \draw (u3)-- (x2);
      \draw (u2)-- (y2);

      \draw (u3) to [out=-60,in=-120, distance=2.5cm] node[draw=none,fill=none,right] {} (y1) ;
      \draw (u2) to [out=-120,in=-60, distance=2.5cm] node[draw=none,fill=none,right] {} (x1) ;
      \draw[dashed] (u2) to [out=-100,in=-80, distance=3.5cm]  (u3);
      \draw[dashed] (u1)-- (u4);

    \end{tikzpicture}
    \caption{Starting configuration and possible extension when $xy$ is in four $C_5$'s.}
    \label{start2}
  \end{center}
\end{figure}

If both $u_1u_4$ and $u_2u_3$ are edges, then the graph is completed and it is the Petersen graph.
\begin{figure}[!ht]
  \begin{center}
\tikzstyle{every node}=[circle, draw, fill=black!50,
                        inner sep=0pt, minimum width=4pt]
 \begin{tikzpicture}[thick,scale=0.8]%
 \draw \foreach \x in {18,90,...,306} {
        (\x:1)  -- (\x+72:1)
		(18:1) node[label=above:$y_2$](y2){}
		(162:1) node[label=left:$x_2$](x2){}
		(90:1) node[label=right:$u_4$](u4){}
		};
   \begin{scope}[shift={(0,-1.618)}]
       \draw \foreach \x in {18,90,...,306} {
        (\x+36:1)  -- (\x-36:1)
        (270:1) node[label=above:$u_1$](u1){}
		(342:1) node[label=right:$y_1$](y1){}
		(198:1) node[label=below:$x_1$](x1){}
		(54:1) node[label=right:$y$](y){}
		(126:1) node[label=left:$x$](x){}
        };
      \end{scope}
      \draw (-2.2,-1.8) node[label=left:$u_3$] (u3){};
      \draw (2.2,-1.8) node[label=right:$u_2$] (u2){};
      \draw (u3)-- (x2);
      \draw (u2)-- (y2);

      \draw (u3) to [out=-60,in=-120, distance=2.5cm] node[draw=none,fill=none,right] {} (y1) ;
      \draw (u2) to [out=-120,in=-60, distance=2.5cm] node[draw=none,fill=none,right] {} (x1) ;
      \draw (u2) to [out=-100,in=-80, distance=3.5cm]  (u3);
      \draw (u1)-- (u4);
    \end{tikzpicture}\hfil
 \begin{tikzpicture}[thick,scale=0.8]%
 \draw \foreach \x in {18,90,...,306} {
        (\x:1)  -- (\x+72:1)
		(18:1) node[label=above:$y_2$](y2){}
		(162:1) node[label=left:$x_2$](x2){}
		(90:1) node[label=right:$u_4$](u4){}
		};
   \begin{scope}[shift={(0,-1.618)}]
       \draw \foreach \x in {18,90,...,306} {
        (\x+36:1)  -- (\x-36:1)
        (270:1) node[label=above:$u_1$](u1){}
		(342:1) node[label=right:$y_1$](y1){}
		(198:1) node[label=below:$x_1$](x1){}
		(54:1) node[label=right:$y$](y){}
		(126:1) node[label=left:$x$](x){}
        };
      \end{scope}
      \draw (-2.2,-1.8) node[label=left:$u_3$] (u3){};
      \draw (2.2,-1.8) node[label=left:$u_2$] (u2){};
      \draw (u3)-- (x2);
      \draw (u2)-- (y2);

      \draw (u3) to [out=-60,in=-120, distance=2.5cm] node[draw=none,fill=none,right] {} (y1) ;
      \draw (u2) to [out=-120,in=-60, distance=2.5cm] node[draw=none,fill=none,right] {} (x1) ;

      \draw (u1)-- (u4);

      \draw (3.2,-1.8) node[label=right:$z$] (z){};
      \draw (u2) -- (z);
    \end{tikzpicture}\\
    \begin{minipage}{0.45\linewidth}
    \caption{The Petersen Graph.}
    \end{minipage}\hfil
    \begin{minipage}{0.45\linewidth}
    \caption{The edge $zu_2$ is not in any $C_5$. Contradiction!}
    \end{minipage}
  \end{center}
\end{figure}

If only one of them is an edge, by symmetry, we can assume $u_1u_4$ is an edge but $u_2u_3$ are not.
Then $u_2$ must have an new neighbor, called $z$. Now the edge $zu_2$ can not be in any $C_5$.
Otherwise, say $zu_2XYZz$ is the $C_5$. We must have $X\in \{x_1, y_2\}$, $Y\in \{x,u_1,y, u_4\}$,
and $Z\in\{x_2, y_1\}$. But now $Zz$ is not an edge. Contradiction!
\end{proof}

\end{document}